\documentclass[final]{siamltex}
\usepackage{latexsym,bm,amssymb,amsfonts,amsbsy,amsmath,graphics,graphicx,color}
\usepackage{url}
\newtheorem{rem}[theorem]{Remark}

\def\RR{\mathbb R}

\def\pmatrix{ \left( \begin{array} }
\def\endpmatrix{ \end{array} \right) }

\def\d2dxx{\frac{\partial^2}{\partial x^2}}

\def\no{\noindent}

\def\phi{\varphi}

\def\P{{\cal P}}

\def\A{{\cal A}}

\title{On the existence of energy-preserving symplectic integrators based upon Gauss collocation formulae
\thanks{Work developed within the project ``Numerical methods and software for differential
equations''.}}

\author{Luigi Brugnano\thanks{Dipartimento di Matematica ``U.\,Dini'', Universit\`a di
Firenze, Italy ({\tt luigi.brugnano@unifi.it}).} \and Felice
Iavernaro\thanks{Dipartimento di Matematica, Universit\`a di Bari,
Italy ({\tt felix@dm.uniba.it}).} \and Donato
Trigiante\thanks{Dipartimento di Energetica, Universit\`a di
Firenze, Italy ({\tt trigiant@unifi.it}).}}

\begin{document}

\maketitle

\begin{abstract}
We introduce a new family of symplectic integrators depending on a
real parameter $\alpha$. For $\alpha=0$, the corresponding method in
the family becomes the classical Gauss collocation formula of order
$2s$, where $s$ denotes the number of the internal stages. For any
given non-null $\alpha$, the corresponding method remains symplectic
and has order $2s-2$: hence it  may be interpreted as a
$O(h^{2s-2})$ (symplectic) perturbation of the Gauss method. Under
suitable assumptions, we show that the parameter $\alpha$ may be
properly tuned, at each step of the integration procedure, so as to
guarantee energy conservation in the numerical solution. The
resulting  method shares the same order
 $2s$ as the generating Gauss formula.
\end{abstract}



\pagestyle{myheadings} \thispagestyle{plain}

\markboth{L. BRUGNANO, F. IAVERNARO AND D.
TRIGIANTE}{ENERGY-PRESERVING, SYMPLECTIC INTEGRATORS}

\section{Introduction}

We consider  canonical  Hamiltonian systems in the form
\begin{equation}\label{hamilode}
\left\{  \begin{array}{l} \dot y =  J\nabla H(y) \equiv f(y),  \\
y(t_0) = y_0 \in\RR^{2m}, \end{array} \right.
 \qquad J=\pmatrix{rr} 0 & I \\ -I & 0 \endpmatrix \in \RR^{2m \times
 2m},
\end{equation}
($I$ is the identity matrix of dimension $m$). Regarding its
numerical integration, two main lines of investigation may be
traced, having as objective the definition and the study of
symplectic methods and energy-conserving methods, respectively. In
fact, symplecticity and the conservation of the energy function
are the most relevant features characterizing a Hamiltonian
system.

 From the very beginning of this research activity, high order symplectic
formulae were already available within the class of Runge-Kutta
methods,  the Gauss collocation formulae being one noticeable
example. One important implication of symplecticity of the
discrete flow is the conservation of quadratic invariants. This
circumstance makes the symplecticity property of a method
particularly appealing in the numerical simulation of isolated
mechanical systems in the form \eqref{hamilode}, since it provides
a precise conservation of the total angular momentum during the
time evolution of the state vector. As a further positive
consequence, a symplectic method also conserves quadratic
Hamiltonian functions (see the monographs \cite{HLW,LR,SC} for a
detailed analysis of symplectic methods).

On the other hand, excluding the quadratic case, energy-conserving
methods were initially not known within the class of classical
methods and as a matter of facts, among the first attempts to
address this issue, projection and symmetric projection techniques
were coupled to classical non-conservative schemes in order to
impose the numerical solution to lie in a proper manifold
representing a first integral of the original system (see
\cite[Sect. VII.2]{HW}, \cite{AR,H} and \cite[Sect. V.4.1]{HLW}).

A completely new approach is represented by {\em discrete gradient
methods} which are based upon the definition of a discrete
counterpart of the gradient operator so that energy conservation of
the numerical solution is guaranteed at each step and whatever the
choice of the stepsize of integration (see \cite{G,MQR}).

More recently, the conservation of energy has been approached by
means of the definition of the {\em discrete line integral}, in a
series of papers (such as \cite{IP1,IT3}), leading to the definition
of {\em Hamiltonian Boundary Value Methods (HBVMs)} (see for example
\cite{BIT,BIT1}). These are a class of methods able to preserve, in
the discrete solution, polynomial Hamiltonians of arbitrarily high
degree (and hence, a {\em practical} conservation of any
sufficiently differentiable Hamiltonian)\footnote{We refer the
reader to \cite{BIT0} for a complete documentation on HBVMs.}. Such
methods admit a Runge-Kutta formulation which reveals their close
relationship with classical collocation formulae \cite{BIT3}. An
infinity extension of HBVMs has also been proposed in
\cite{Ha,BIT1}.

Attempts to incorporate both symplecticity and energy conservation
into the numerical method will clash with two non-existence
results. The first \cite{GM} refers to non-integrable systems,
that is systems that do not admit  other independent first
integrals different from the Hamiltonian function itself.
According to the authors' words, it states that
\begin{quote} \em
If [the method] is symplectic, and conserved $H$ exactly, then it is
the time advance map for the exact Hamiltonian system up to a
reparametrization of time.
\end{quote}
The second negative result \cite{CFM} refers to B-series symplectic
methods applied to general (not necessarily non-integrable)
Hamiltonian systems:
\begin{quote} \em
The only symplectic method (as $B$-series) that conserves the
Hamiltonian for arbitrary $H(y)$ is the exact flow of the
differential equation.
\end{quote}

The aim of the present work is to devise methods of any high order
that, in a sense that will be specified below and under suitable
conditions, may share both features.
More precisely, we will begin with introducing  a family of one-step
methods
\begin{equation}
\label{met_alpha}
y_1(\alpha)=\Phi_h(y_0,\alpha)
\end{equation}
($h$ is the stepsize of integration), depending on a real parameter
$\alpha$, with the following specifics:
\begin{enumerate}
\item for any fixed choice of $\alpha \not = 0$, the corresponding method is a
symplectic Runge-Kutta method with $s$ stages and of order $2s-2$;
\item for $\alpha=0$ one gets the Gauss collocation method (of order $2s$);
\item for any choice of $y_0$ and in a given range of the stepsize $h$,
there exists a value of the parameter, say $\alpha^\ast$, depending
on $y_0$ and $h$, such that $H(y_1)=H(y_0)$ (energy conservation).
\end{enumerate}
As the parameter $\alpha$ ranges in a small interval centered at
zero, the value of the numerical Hamiltonian function $H(y_1)$ will
match $H(y(t_0+h))$, thus leading to energy conservation. This
result, which will be formally proved in Section \ref{theory}, is
formalized as follows:

\smallskip

\textit{Under suitable assumptions, there exists a real sequence
$\{\alpha_k\}$ such that the numerical solution defined by
$y_{k+1}=\Phi_h(y_{k},\alpha_k)$, with $y_0$ defined in
(\ref{hamilode}), satisfies $H(y_k)=H(y_0)$. }

\smallskip

To clarify this statement and how it relates to the above
non-existence results, we emphasize that the energy conservation
property only applies to the  specific numerical orbit $\{y_k\}$
that the method generates, starting from the initial value $y_0$ and
with stepsize $h$. For example, let us consider the very first step
and assume the existence of a value $\alpha=\alpha_0$, in order to
enforce the energy conservation between the two state vectors $y_0$
and $y_1$, as indicated at item 3 above. If $\alpha_0$ is maintained
constant, the map $y \mapsto \Phi_h(y,\alpha_0)$ is symplectic and,
by definition, assures the energy conservation condition
$H(y_1)=H(y_0)$. However,  it would fail to provide a conservation
of the Hamiltonian function if we changed the initial condition
$y_0$ or the stepsize $h$: in general, for any $\hat y_0 \not =y_0$,
we would obtain $H(\Phi_h(\hat y_0,\alpha_0)) \not = H(y_0)$. Thus,
the energy conservation property we are going to discuss weakens the
standard energy conservation condition mentioned in the two
non-existence results stated above and hence, by no means, the new
methods are  meant to produce a counterexample of these statements.

The paper is organized as follows. In the next section we report
the definition of the methods while in Section \ref{colloc} we
show their geometrical link with Gauss collocation formulae. In
Section \ref{theory} we face the problem from a theoretical
viewpoint and give some existence results that aim to explain the
energy-preserving property of the new methods. In Section
\ref{numerical_tests} we report a few
 tests that give a clear numerical evidence that a change
in sign of the function $g(\alpha)=H(y_1(\alpha))-H(y_0)$ does
indeed occur along the integration procedure.

\section{Definition of the methods}
\label{definition} Let $c_1<c_2<\dots<c_s$ and $b_1,\dots, b_s$ be
the abscissae and the weights of the Gauss-Legendre quadrature
formula in the interval $[0,1]$. We consider the Legendre
polynomials $P_j(\tau)$, of degree $j-1$ for $j=1,\dots,s$, shifted
and  normalized in the interval $[0,1]$, that is
\begin{equation}
\label{orth} \int_0^1 P_i(\tau)P_j(\tau) \mathrm{d} \tau =
\delta_{ij}, \qquad i,j=1,\dots,s,
\end{equation}
($\delta_{ij}$ is the Kronecker symbol), and the matrix
\begin{equation}
\label{P} \P = \pmatrix{cccc}
P_1(c_1) & P_2(c_1) & \cdots & P_s(c_1) \\
P_1(c_2) & P_2(c_2) & \cdots & P_s(c_2) \\
\vdots   & \vdots   &        & \vdots \\
P_1(c_s) & P_2(c_s) & \cdots & P_s(c_s)
\endpmatrix_{s \times s}.
\end{equation}
Our starting point is the following decomposition of the Butcher
array $A$ of the Gauss method of order $2s$ (see \cite[Theorem
5.6]{HW}):
\begin{equation}\label{A}
A= \P X_s \P^{-1},
\end{equation}
where $X_s$ is defined as
\begin{equation}\label{Xs}
X_s = \pmatrix{cccc}
\frac{1}2 & -\xi_1 &&\\
\xi_1     &0      &\ddots&\\
          &\ddots &\ddots    &-\xi_{s-1}\\
          &       &\xi_{s-1} &0\\
\endpmatrix, \end{equation}
\no with
\begin{equation}
\label{xij}\xi_j=\frac{1}{2\sqrt{(2j+1)(2j-1)}}, \qquad
j=1,\dots,s-1.\end{equation}

We now consider the matrix $X_s(\alpha)$ obtained by perturbing
\eqref{Xs} as follows:
\begin{equation}\label{Xs_alpha}
X_s(\alpha) = \pmatrix{cccc}
\frac{1}2 & -\xi_1 &&\\
\xi_1     &0      &\ddots&\\
          &\ddots &\ddots    &-(\xi_{s-1}+\alpha)\\
          &       &\xi_{s-1}+\alpha &0\\
\endpmatrix = X_s + \alpha W_s,\end{equation}
where $\alpha$ is a real parameter, and
\begin{equation}\label{Ws}
W_s=\pmatrix{cccc}
0 & 0 &&\\
0     &0      &\ddots&\\
          &\ddots &\ddots    &-1\\
          &       &1 &0\\
\endpmatrix,
\end{equation}
so that $X_s(\alpha)$ is a rank two perturbation  of $X_s$.

The family of methods  $y_1=\Phi_h(y_0,\alpha)$ we are interested
in, is defined by the following tableau:
\begin{equation}
\label{qgauss}
\begin{array}{c|c}\begin{array}{c} c_1\\ \vdots\\ c_s\end{array} & \A(\alpha) \equiv \P X_s(\alpha)\P^{-1}\\
 \hline                    &b_1\, \ldots \ldots ~ b_s
\end{array}
\end{equation}
Therefore
\begin{equation}
\label{BA}
\A(\alpha)= A + \alpha \P W_s \P^{-1},
\end{equation}
and hence $\A(0)=A$.

By exploiting Theorems~5.11 and 5.1 in \cite[Chap. IV.5]{HW}, we
readily deduce that the symmetric method \eqref{qgauss} has order
$2s-2$ for any fixed $\alpha \not = 0$, and order $2s$ when
$\alpha=0$.

We set
\begin{equation}\label{Omega}
\omega =\pmatrix{c} b_1 \\ \vdots \\ b_s \endpmatrix, \qquad \Omega
=\pmatrix{ccc} b_1 \\ & \ddots \\&&  b_s \endpmatrix, \qquad e=
\pmatrix{c} 1 \\ \vdots \\ 1 \endpmatrix.
\end{equation}

\begin{theorem}
\label{symplecticity} For any value of $\alpha$,  the Runge-Kutta
method defined in \eqref{qgauss} is symplectic.
\end{theorem}

\begin{proof}
On the basis of \cite[Theorem 4.3, page 192]{HLW}, we will prove the
following sufficient condition for symplecticity:
$$
\Omega \A(\alpha) + \A(\alpha)^T \Omega = \omega \, \omega^T.
$$
Since the degree of the integrand functions in  \eqref{orth} does
not exceed $2s-2$, the ortho\-gonality conditions may be
equivalently posed in discrete form as
$$
\sum_{k=1}^s b_k P_i(c_k)P_j(c_k) = \delta_{ij}, \qquad
i,j=1,\dots,s,
$$
or, in matrix notation,
\begin{equation}
\label{orth1} \P^T \Omega P = I.
\end{equation}
Considering that from \eqref{orth1} we get $\P^{-1}= \P^T \Omega$,
from \eqref{BA} we have that
\begin{equation}
\label{hint} \Omega \A(\alpha) + \A(\alpha)^T \Omega = \Omega A +
A^T \Omega + \alpha \Omega \P ( W_s + W_s^T) \P^T \Omega = \omega \,
\omega^T
\end{equation}
since the Gauss method is symplectic, and $W_s$ is skew-symmetric
so that $W_s+W_s^T=0$. \qquad
\end{proof}

In the event that  a value $\alpha^\ast \equiv \alpha^\ast(y_0,h)$
for the parameter $\alpha$ may be found such that the conservation
condition $ H(y_1(\alpha))=H(y_0)$ be satisfied, we can extrapolate
from the parametric method \eqref{qgauss} a  symplectic scheme
\begin{equation}
\label{epgauss} y \mapsto \Phi_h(y,\alpha^\ast),
\end{equation}
that provides energy conservation if evaluated at $y_0$. The
existence of such an $\alpha^\ast$ will be proved in Section
\ref{theory}. One important implication the use of \eqref{epgauss}
will guarantee is the conservation of all quadratic constant of
motions associated with system \eqref{hamilode}. The fact that, in
general, $H(\Phi_h(y,\alpha^\ast))\not = H(y)$, explains the
extent to which the energy conservation property of the new
formulae must be interpreted. Summarizing, the new formulae, when
applied to the initial value system \eqref{hamilode} are able to
define a numerical approximation of any high order, along which
the Hamiltonian function and all quadratic first integrals of the
system are precisely conserved.

\subsection{Generalizations}
The proof of Theorem \ref{symplecticity} suggests how to extend the
definition of the new formulae in order to get a
 family of methods depending on a set of parameters. Indeed, by
 looking at \eqref{hint}, in order to preserve symplecticity,
it is sufficient to substitute to the matrix $\alpha W_s$, any
skew-symmetric matrix $\widetilde W_s$ of low rank, having non-null
elements in the bottom-right corner so that, with respect to the
Gauss method, the order is lowered as least as possible.
\begin{theorem}
 Consider the $s\times s$ matrix
$$
\widetilde W_s= \pmatrix{ll} 0 \\ & V_r \endpmatrix
$$
where $V_r$ is any skew-symmetric matrix of dimension $r+1<s$. The
Runge-Kutta method defined by the Butcher tableau
\begin{equation}
\label{qgauss1}
\begin{array}{c|c}\begin{array}{c} c_1\\ \vdots\\ c_s\end{array} & \A(\alpha) \equiv \P (X_s+\widetilde W_s)\P^{-1}\\
 \hline                    &b_1\, \ldots \ldots ~ b_s \end{array}
\end{equation}
in \eqref{qgauss} is symplectic and has order $p=2(s-r)$.
\end{theorem}

For example, a natural choice for the matrix $\widetilde W_s$ is:
\begin{equation}
\label{tildeWs} \widetilde W_s=\pmatrix{cccccc}
0 & 0 &&\\
0     &0      &\ddots&\\
          &\ddots &\ddots    &-\alpha_1\\
          &       &\alpha_1 &0 & \ddots \\
          &       &  & \ddots & \ddots & -\alpha_r\\
          &       &  &       & \alpha_r & 0
\endpmatrix
\end{equation}
leading to a multi-parametric method depending on the $r$ parameters
$\alpha_1,\dots,\alpha_r$.

\section{Quasi-collocation conditions}\label{colloc}
Condition \eqref{BA} reveals the relation between the Butcher arrays
associated with the new parametric method and the Gauss collocation
method. In order not to loose generality, just in this subsection we
assume to solve the generic problem $\dot y = f(y)$.

We wander how the collocation conditions defining the Gauss methods
are affected by the presence of the parameter $\alpha$. This is
easily accomplished by expressing the coefficients of the perturbing
matrix $\P W_s \P^{-1}$ in terms of linear combinations of the
integrals $\int_0^{c_i} l_j(\tau)\mathrm{d} \tau$, where $l_j(\tau)$
is the $j$th Lagrange polynomial defined on the abscissae
$c_1,\dots,c_s$. Let $\Gamma\equiv\left(\gamma_{ij}\right)$ be the
solution of the matrix linear system $A \, \Gamma = \P W_s \P^{-1}$,
which means that (see \eqref{A})
\begin{equation}
\label{Gamma} \Gamma = \P X_s^{-1} W_s \P^{-1}.
\end{equation}
The nonlinear system defining the block vector of the internal
stages $\{Y_i\}$ is
$$
Y = e \otimes y_0 + h (A\otimes I) F(Y) + \alpha h (A\,\Gamma
\otimes I) F(Y),
$$
where $e$ is the vector defined in (\ref{Omega}), hereafter $I$ is
the identity matrix of dimension $2m$, and
$$Y=\pmatrix{ccc} Y_1^T & \dots & Y_s^T \endpmatrix^T, \qquad
F(Y)=\pmatrix{ccc} f(Y_1)^T & \dots & f(Y_s)^T \endpmatrix^T.$$
Therefore, the polynomial $\sigma(t_0+\tau h)$ of degree $s$ that
interpolates the stages $Y_i$ at the abscissae $c_i$,
$i=1,\dots,s,$ is
\begin{equation}
\label{sigma} \displaystyle  \sigma(t_0+\tau h) = y_0 + h
\sum_{j=1}^s \int_0^\tau l_j(x)\mathrm{d}x \, f(Y_j) + \alpha h
\sum_{j=1}^s  \left( \sum_{k=1}^s \gamma_{kj} \int_0^\tau
l_k(x)\mathrm{d}x \right)\, f(Y_j).
\end{equation}
Differentiating \eqref{sigma} with respect to $\tau$ gives
\begin{equation}
\label{sigma_dot}
\displaystyle \dot \sigma(t_0+\tau h) =
\sum_{j=1}^s l_j(\tau)\, f(\sigma(t_0+c_j h)) + \alpha \sum_{j=1}^s
\left( \sum_{k=1}^s \gamma_{kj} l_k(\tau) \right)\, f(\sigma(t_0+c_j
h)).
\end{equation}
Finally, evaluating \eqref{sigma} at $\tau=0$ and \eqref{sigma_dot}
at $\tau=c_i$ yields
\begin{equation}
\label{coll} \left\{ \begin{array}{l}\displaystyle \sigma(t_0)  =
y_0,
\\ \displaystyle \dot \sigma(t_0+c_i h)  = f(\sigma(t_0+c_i h)) +
\alpha \sum_{j=1}^s \gamma_{ij} \, f(\sigma(t_0+c_j h)), \qquad
i=1,\dots,s. \end{array}
\right.
\end{equation}
For $\alpha$ small, we can regard \eqref{coll} as {\em
quasi-collocation conditions}, since for $\alpha=0$ we recover the
classical collocation conditions defining the Gauss method.

\subsection{Geometric interpretation}
Let us assume the existence of a quadratic first integral $M(y)$
independent from $H(y)$: although this assumption is not strictly
needed,
 it will somehow simplify the presentation of our argument.
\begin{figure}[htb]
\centerline{\includegraphics[width=12cm,height=7cm]{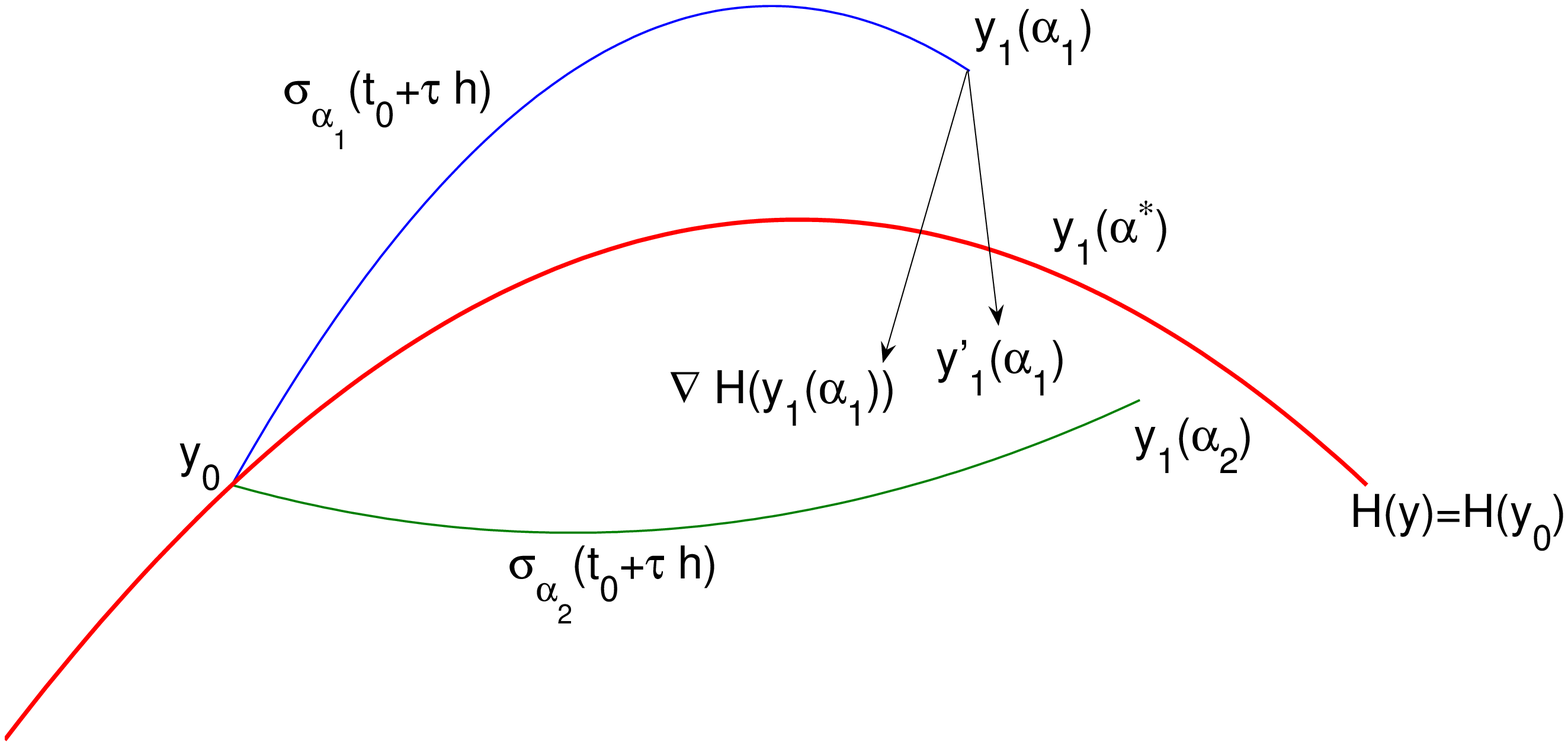}}
\caption{A geometric interpretation of the parametric method
\eqref{qgauss}. Two quasi-collocation polynomials
$\sigma_{\alpha_1}(t_0+\tau h)$ and $\sigma_{\alpha_2}(t_0+\tau h)$
have as end-points the numerical solutions $y_1(\alpha_1)$ and
$y_1(\alpha_2)$ which are $O(h^{2s-1})$ close to the Hamiltonian
since, for $\alpha \not = 0$ the method has order $2s-2$. This means
that the length of the arc of curve enclosed by the points
$y_1(\alpha_1)$ and $y_1(\alpha_2)$ is $O(h^{2s-1})$. However, the
parametric curve $\gamma: \alpha\in [\alpha_1,\alpha_2] \mapsto
y_1(\alpha)$ passes through $y_1(0)$ which is at a distance
$O(h^{2s+1})$ from the manifold $H(y)=H(y_0)$, and there is a
concrete possibility that this arc may intersect the manifold
$H(y)=H(y_0)$ at a point $y_1(\alpha^\ast)$.} \label{sliding}
\end{figure}

Roughly speaking, for $\alpha$ small, our {\em parametric} method
may be interpreted as a symplectic perturbation of the Gauss
method. Due to symplecticity of $\Phi_h(\cdot,\alpha)$, the
parametric curve
\begin{equation}
\label{gamma} \gamma \equiv \alpha \in D \mapsto y_1(\alpha)\in
\RR^{2m},
\end{equation}
where $D$ is a given interval containing zero, will entirely lie
in the manifold $M(y)=M(y_0)$ and its length will be
$O(h^{2s-1})$, since the method has order $2s-2$. However, the
numerical solution produced by the Gauss method, namely $y_1(0)$
will be $O(h^{2s+1})$ close to the manifold $H(y)=H(y_0)$. Since
the two manifolds contain the continuous solution, their
intersection is nonempty and it is reasonable to expect that when
$\alpha$ ranges in $D$, $y_1(\alpha)$ can slide from a region
where $H(y_1(\alpha))>H(y_0)$ to a region where
$H(y_1(\alpha))<H(y_0)$, thus producing a sign change in the
scalar function
\begin{equation}
\label{g} g(\alpha)= H(y_1(\alpha))-H(y_0)
\end{equation}
which, by continuity, will vanish  at a point $\alpha^\ast$.

Obviously, similar arguments can be repeated for the multi-parameter
version (\ref{qgauss1}) of the method, where one has even more
freedom in the choice of the parameters in the matrix $\widetilde
W_s$ defined in \eqref{tildeWs}, in order to obtain the conservation
of energy.\footnote{We do not consider multi-parametric methods in
the numerical results we present, since a single parameter suffices
in getting the energy conservation property.}

\section{Theoretical existence results}
\label{theory} After defining the error function $g(\alpha)=
H(y_1(\alpha))- H(y_0)$, the nonlinear system, in the unknowns
$Y_1,\dots,Y_s$ and $\alpha$, that is to be solved at each step for
getting energy conservation, reads
\begin{equation}
\label{concon} \left\{ \begin{array}{l} Y = e \otimes y_0 + h
(\A(\alpha) \otimes I) F(Y), \\
g(\alpha) =0,
\end{array}
\right.
\end{equation}
and its solvability is equivalent to the existence  of the energy
preserving method \eqref{epgauss} we are looking for. After defining
the vector function
$$
G(h,y_1,\alpha)=\pmatrix{c} y_1-\Phi_h(y_0,\alpha) \\[.2cm]
H(y_1)-H(y_0) \endpmatrix,
$$
we see that system \eqref{concon} is equivalent to
$G(h,y_1,\alpha)=0$. Of course $G(0,y_1,\alpha)=0$ for any value of
$\alpha$ and, in particular $G(0,y_1,0)=0$. The Jacobian of $G$ with
respect to the two variables $y_1$ and $\alpha$ reads
$$
\frac{\partial G}{\partial(y_1,\alpha)} (h,y_1,\alpha) =
\pmatrix{cc} I &
\frac{\partial \Phi_h}{\partial \alpha}(y_0,\alpha) \\[.2cm]
\nabla^TH(y_1) & 0 \endpmatrix,
$$
where, as usual, $I$ is the identity matrix of dimension $2m$.  From
\eqref{epgauss} we see that $\frac{\partial \Phi_h}{\partial
\alpha}(y_0,\alpha)$ coincides with $y'_1(\alpha)$ and, hence, with
$\sigma_{\alpha}'(t_0+h)$.  Due to the consistency of the method, it
follows that, for $\alpha=0$, $\sigma_{\alpha}'(t_0+h) \rightarrow J
\nabla H(y_0)$ as $h\rightarrow 0$. Therefore
\begin{equation}
\label{Gjac} \frac{\partial G}{\partial(y_1,\alpha)} (0,y_1,0) =
\pmatrix{cc} I &
J \nabla H (y_0) \\[.2cm]
\nabla^TH(y_0) & 0 \endpmatrix.
\end{equation}
Unfortunately, the Jacobian matrix \eqref{Gjac} is always singular.
Consequently, the implicit function theorem (in its classical
formulation) does not help in retrieving existence results of the
solution of \eqref{concon} when $h$ is small. However, the rank of
the matrix \eqref{Gjac} is $2m$ independently of the problem to be
solved. This would suggest the use of the Lyapunov-Schmidt
decomposition \cite{SLSF} that considers the restriction of the
system to both the complement of the null space and the range of the
Jacobian, to produce two systems to which the implicit function
theorem applies.

In our case this approach is simplified in that the implicit
function theorem assures the existence of a solution $Y(\alpha)$ of
the first system in \eqref{concon} for all values of the parameter
$\alpha$ ranging in a closed interval containing the origin and $|h|
\le h_0$, with $h_0$ small enough. Then $y_1(\alpha)=y_0+h(b^T
\otimes I)Y(\alpha)$ is substituted into the second of
\eqref{concon} to produce the so called {\em bifurcation equation}
in the unknown $\alpha$. When needed, we will explicitly write
$g(\alpha, h)$ or $g(\alpha, h, y_0)$, in place of $g(\alpha)$, to
emphasize the dependence of the function $g$ upon the stepsize $h$,
that has to be treated as a parameter, and the state vector $y_0$.

Let us fix a vector $y_0$ and look for solution curves of
$g(\alpha,h)=0$ in the $(h,\alpha)$ plane. Obviously $g(\alpha,0)=0$
for any $\alpha$, which means that the axis $h=0$ is a solution
curve of the bifurcation equation: of course, we are interested in
the existence of a different solution curve
$\alpha^\ast=\alpha^\ast(h)$ passing through the origin. Since the
gradient of $g$ vanishes at $(0,0)$, one has to compute the
subsequent partial derivatives of $g$ with respect to $\alpha$ and
$h$. However one verifies that $\frac{\partial^2 g}{\partial h^2} =
\frac{\partial^2 g}{\partial \alpha^2}=\frac{\partial^2 g}{\partial
\alpha \partial h}$ evaluated at $(0,0)$ vanish as well, and this
makes the computations even harder. For this reason, to address the
question about the existence of a solution of \eqref{concon}, we
make the following assumptions:
\begin{itemize}
\item[($\mathcal A_1$)] the function $g$ is analytical in a rectangle $[-\bar \alpha, \bar \alpha] \times [-\bar
h, \bar h]$ centered at the origin;
\item[($\mathcal A_2$)] let $d$ be the order of the error in the Hamiltonian function
associated with the Gauss method applied to the given Hamiltonian
system \eqref{hamilode} and the given state vector $y_0$, that is:
\begin{equation}
\label{gH} g(0,h) = H(y_1(0))- H(y_0) = c_0 h^{d} + O(h^{d+1}),
\end{equation}
with $c_0 \not = 0$. Then, we assume that for any fixed $\alpha
\not = 0$, $$g(\alpha,h) = c(\alpha) h^{d-2} + O(h^{d-1}),$$ with
$c(\alpha) \not =0$.
\end{itemize}
\begin{rem}
A couple of quick comments are in order before continuing. Excluding
the case where the Hamiltonian $H(q,p)$ is quadratic (which would
imply $g(\alpha,h) = 0$ for all $alpha$), the error  in the
numerical Hamiltonian function associated with the Gauss method is
expected to behave as $O(h^{2s+1})$. Anyway, we cannot exclude a
priori that special classes of problems or particular values for the
state vector $y_0$ may occur, for which the order of convergence may
be even higher. This is why we have introduced the integer $d$:
therefore such integer will be at least $2s+1$. Moreover, we
emphasize that the constant $c_0$ and the function $c(\alpha)$, will
depend on $y_0$. In conclusion, what we are assuming is that for the
method \eqref{qgauss}, when $\alpha$ is a given nonzero constant,
the order of the error $H(y_1(\alpha))-H(y_0)$ is lowered by two
units with respect to the underlying Gauss method of order $2s$,
which is a quite natural requirement since such method has order
$2s-2$.
\end{rem}

\begin{theorem} \label{implicit}
Under the assumptions ($\mathcal A_1$) and ($\mathcal A_2$), there
exists a function $\alpha^\ast=\alpha^\ast(h)$, defined in a
neighborhood of the origin  $(-h_0,h_0)$,   such that:
\begin{itemize}
\item[(i)] $g(\alpha^\ast(h),h)=0$, for all $h\in(-h_0,h_0)$,
\item[(ii)]$\alpha^\ast(h)=\mathrm{const}\cdot h^2 + O(h^3)$.
\end{itemize}
\end{theorem}

\begin{proof}
 From ($\mathcal A_1$) and ($\mathcal A_2$) we obtain that the
 expansion of $g$ around $(0,0)$ is:
\begin{equation}
\label{expg} g(\alpha,h)= \sum_{j=d}^\infty
\frac{1}{j!}\frac{\partial^{j}g}{\partial h^j} (0,0) h^j  +
\sum_{i=1}^{\infty} \sum_{j=d-2}^\infty
\frac{1}{i!j!}\frac{\partial^{i+j}g}{\partial \alpha^i
\partial h^j} (0,0) h^j \alpha^i.
\end{equation}
We are now in the right position to apply the implicit function
theorem. We will look for a solution $\alpha^\ast =
\alpha^\ast(h)$ in the form $\alpha^\ast(h)=\eta(h) h^2$, where
$\eta(h)$ is a real-valued function of $h$. To this end, we
consider the change of variable $\alpha=\eta h^2$, and insert it
into \eqref{expg} thus obtaining
\begin{equation}
\label{expg1} \begin{array}{rl} \displaystyle g(\alpha,h) = &
\displaystyle \frac{1}{d!}\frac{\partial^{d}g}{\partial h^d} (0,0)
h^d  + \frac{1}{(d-2)!}\frac{\partial^{d-1}g}{\partial \alpha
\partial h^{d-2}} (0,0) h^d \eta \\[.3cm] &  \displaystyle
+ \frac{1}{(d-1)!}\frac{\partial^{d}g}{\partial \alpha
\partial h^{d-1}} (0,0) h^{d+1} \eta  + \mbox{higher order terms}.
\end{array}
\end{equation}
Therefore, for $h\not=0$, $g(\alpha,h)=0$ is equivalent to $\tilde
g(\eta, h)=0$, where
\begin{equation}
\label{expg2} \begin{array}{rl} \displaystyle \tilde g(\eta,h) = &
\displaystyle \frac{1}{(d-1)d}\, \frac{\partial^{d}g}{\partial h^d}
(0,0)   + \frac{\partial^{d-1}g}{\partial \alpha
\partial h^{d-2}} (0,0)  \eta \\[.3cm] &  \displaystyle
+ \frac{1}{d-1}\frac{\partial^{d}g}{\partial \alpha
\partial h^{d-1}} (0,0) h \eta  + \mbox{higher order terms}.
\end{array}
\end{equation}
By assumption ($\mathcal A_2$), both $\frac{\partial^{d}g}{\partial
h^d} (0,0)$ and $\frac{\partial^{d-1}g}{\partial \alpha
\partial h^{d-2}} (0,0)$ are different from zero and hence the
implicit function theorem assures the existence of a function
$\eta=\eta(h)$ such that $\tilde g(\eta(h), h)=0$. The solution of
$g(\alpha,h)=0$ for the variable $\alpha$ will then be given by
\begin{equation}
\label{alphah} \alpha^\ast(h) = \eta(h) h^2 = -\frac{1}{(d-1)d}\,
\frac{\frac{\partial^{d}g}{\partial h^d}
(0,0)}{\frac{\partial^{d-1}g}{\partial \alpha
\partial h^{d-2}} (0,0)}\, h^2 +O(h^3),
\end{equation}
and this completes the proof. \qquad
\end{proof}

By exploiting \cite[Theorem 6.1.2]{KP}, we see that the function
$\alpha^\ast(h)$ is analytic if the power series \eqref{expg} is
absolutely convergent for $|h|\le h_0$ and $|\alpha| \le \alpha_0$.
In any event, the function $\alpha^\ast(h)$ is tangent to the
$h$-axis at the origin which means that a very small correction of
the Gauss method is needed when the stepsize is small enough. As a
matter of fact, the needed correction is so small that the resulting
 method \eqref{epgauss} has indeed order
$2s$ instead of $2s-2$, just as the Gauss method obtained by posing
$\alpha=0$. This is a consequence of the following result.

\begin{theorem}
\label{fastorder} Consider the parametric method \eqref{qgauss} and
suppose that the parameter $\alpha$ is actually a function of the
stepsize $h$, in such a way that $\alpha(h)=O(h^2)$. Then, the
resulting method has order $2s$.
\end{theorem}

\begin{proof}
Let $y_1(\alpha,h)$ be the solution computed by method
\eqref{qgauss} at time $t_0+h$, starting at $y_0=y(t_0)$, and
consider its expansion with respect to the variable $\alpha$, in a
neighborhood of zero:
$$
y_1(\alpha,h) = y_1(0,h)+ y'(\zeta_\alpha,h) \alpha.
$$
We recall that  $y_1(0,h)$ is the numerical solution provided after
a single step of the Gauss method and hence it is $O(h^{2s+1})$
accurate while, for $\alpha \not = 0$, $y_1(\alpha,h)$ yields an
approximation to the true solution of order $2s-1$. This implies
that $y'(\zeta_\alpha,h)$ is $O(h^{2s-1})$. Consequently,
$$
y_1(\alpha,h) - y(t_0+h) = y_1(0,h)- y(t_0+h)+ y'(\zeta_\alpha,h)
\alpha = O(h^{2s+1}) + \alpha O(h^{2s-1}),
$$
from which we deduce that the error at the left hand side is
$O(h^{2s+1})$ if and only if $\alpha=O(h^2)$.\qquad
\end{proof}

Figure \ref{keplerg} reports the level curves of the function
$g(\alpha,h)$ in a neighborhood of the origin, for the Kepler
problem described in Subsection \ref{keplerproblem} (the vector
$y_0$ has been chosen as in \eqref{keplerinitial}). The tick lines
in the plot correspond to the points $(\alpha, h)$ in the plane
where $g$ vanish. This zero level set consists  of the vertical axis
$h=0$ and of the function $\alpha^\ast(h)$, which splits the region
surrounding the origin into two adjacent subregions where the
function $g$ has clearly opposite sign. Despite the local character
of the above existence result, we see that the branches of the
function $\alpha^\ast(h)$ extend away from the origin. Similar
bifurcation diagrams may be traced starting at different values of
$y_0$ for all the test problems we have considered: this suggests
that, in the spirit of the long-time simulation of dynamical
systems, a quite large stepsize may  be used during the numerical
integration performed by method \eqref{epgauss}.
\begin{figure}[htb]
\begin{center}
\includegraphics[width=14cm,height=8cm]{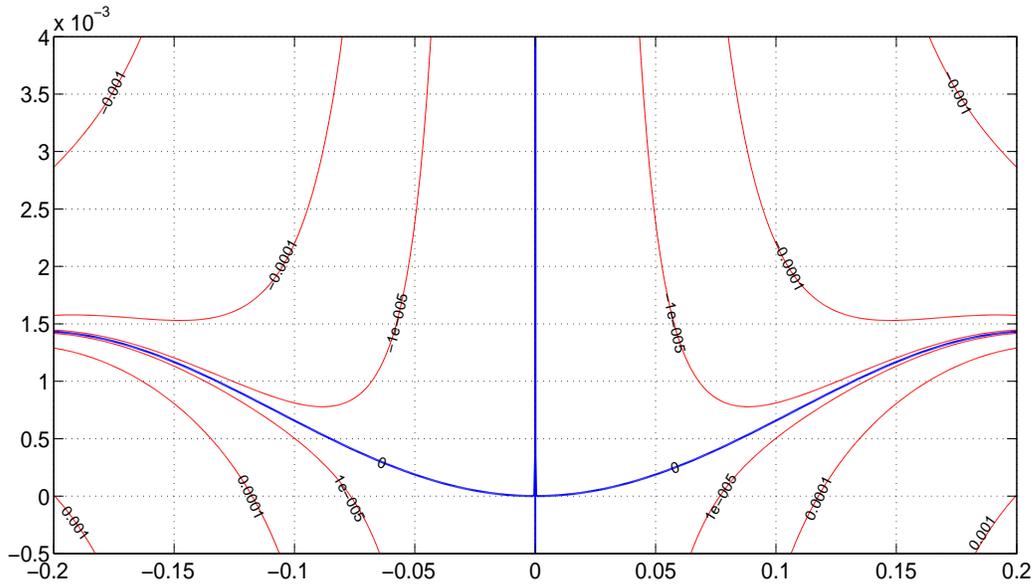}
\end{center}
\caption{Level curves in the plane $(h,\alpha)$ of the function
$g(\alpha,h,y_0)$ associated with the method \eqref{qgauss} of order
four, for the Kepler problem (see Subsection \ref{keplerproblem}),
in a neighborhood of the origin: $h \in [-0.2, 0,2]$, $\alpha \in
[-0.5\cdot 10^{-3}, 4\cdot 10^{-3}]$. Besides the $\alpha$-axis, a
zero level curve tangent to the $h$-axis at the origin is visible.
Such curve separates two regions around the origin where the
function $g$ has opposite sign. We notice that just a small
correction of the Gauss method suffices to recover the energy
preservation even for relatively large stepsizes.} \label{keplerg}
\end{figure}

We end this section by providing a straightforward generalization of
Theorem \ref{implicit} to the case where the parameter $\alpha$ is
used to perturb a generic (not necessarily the last) element on the
subdiagonal of the matrix $X_s$, and its symmetric.

\begin{theorem} \label{implicit1}
Consider the method \eqref{qgauss1} with $\widetilde W_s$ as in
\eqref{tildeWs} with $\alpha_1 \equiv \alpha$ and
$\alpha_2=\dots=\alpha_r=0$. We assume that assumption ($\mathcal
A_1$) and the following assumption (replacing ($\mathcal A_2$))
hold true:
\begin{itemize}
\item[($\mathcal A_2^r$)] let $d$ be the order of the error in the Hamiltonian function
associated with the Gauss method applied to the given Hamiltonian
system \eqref{hamilode} and the given state vector $y_0$. That is,
(\ref{gH}) holds true. Then, we assume that for any fixed $\alpha
\not = 0$, $$g(\alpha,h) = c_r(\alpha) h^{d-2r} + O(h^{d-2r+1}),$$
with $c_r(\alpha) \not =0$.
\end{itemize}
Then, there exists a function $\alpha^\ast=\alpha^\ast(h)$ defined
in a neighborhood of the origin $(-h_0,h_0)$  and such that:
\begin{itemize}
\item[(i)] $g(\alpha^\ast(h),h)=0$, for all $h\in(-h_0,h_0)$,
\item[(ii)]$\alpha^\ast(h)=\mathrm{const}\cdot h^{2r} + O(h^{2r+1})$.
\end{itemize}
The symplectic energy conserving method resulting from this choice
of the parameter has order $2s$.
\end{theorem}

In the next section, we shall provide numerical evidence for the
above presented results.

\section{Numerical tests}
\label{numerical_tests} In this section we present a few numerical
tests showing the effectiveness of our approach. Method
\eqref{epgauss} and its generalization are implemented by solving,
at each step, system \eqref{concon}. The efficient solution of such
system will be the object of future studies; at present we adopt
either one of the following techniques:
\begin{enumerate}
\item at each step, an interval $[\alpha_1, \alpha_2]$ is detected
such that $g(\alpha_1)g(\alpha_2)<0$; after that, a dichotomic
search is implemented to locate $\alpha^\ast$ within an error close
to the machine precision;
\item the first (vector) equation in \eqref{concon} is solved with
$\alpha_0=0$ (Gauss method) and $\alpha_1=c h^r$, where $c$ and $r$
are suitable constants empirically estimated;\footnote{For example
see the last column in Table \ref{kepler_tab}.}
 after that, a sequence
$\alpha_k$ is produced by solving the second (scalar) equation in
\eqref{concon} via the secant method.

\end{enumerate}
In both cases,  an outer iteration generating the sequence
$\alpha_k$ converging to $\alpha^\ast$ is coupled with an
 inner iteration that determines the solution $y_1(\alpha_k)$ starting from
 $y_0$. Such scheme is repeated at each step of integration.

The methods that we will consider in our experiments are: method
\eqref{epgauss} with $s=2$ (fourth order);  method \eqref{epgauss}
with $s=3$ (sixth order); the sixth-order method described in
Theorem \ref{implicit1} with $s=3$, that is we insert a single
perturbation parameter $\alpha$ in the first (rather than in the
second) subdiagonal element of the matrix $X_3$. In order to
distinguish between these two  methods of order six, hereafter the
latter will be referred to as ``the order six method of the second
type''.

\subsection{The Kepler problem} \label{keplerproblem}
In this problem, two bodies subject to Newton's law of gravitation
revolve about their center of mass, placed at the origin, in
elliptic orbits in the $(q_1,q_2)$-plane. Assuming unitary masses
and gravitational constant, the dynamics is described by the
Hamiltonian function
\begin{equation}
\label{Hkepler} H(q_1,q_2,p_1,p_2)= \frac{1}{2}\left( p_1^2+p_2^2
\right) - \frac{1}{\sqrt{q_1^2+q_2^2}}.
\end{equation}
Besides the total energy $H$, a relevant  first integral for the
system is represented by the angular momentum
\begin{equation}
\label{Lkepler} L(q_1,q_2,p_1,p_2) = q_1 p_2 - q_2 p_1.
\end{equation}
Due to its symplecticity, the quadratic first integral
\eqref{Lkepler} will be automatically conserved by method
\eqref{qgauss}, for any choice of the parameter $\alpha$. On the
other hand, we show that, at each step of integration, the parameter
$\alpha$ may be tuned in order to get energy conservation in the
numerical solution.

As initial condition we choose
\begin{equation}
\label{keplerinitial} q_1(0)=1-e, \quad q_2(0)=0, \quad p_1(0)=0,
\quad p_2(0)=\sqrt{\frac{1+e}{1-e}},
\end{equation}
which confers  an eccentricity  equal to $e$ on the orbit.
Consequently, $H(q,p)=-0.5$ and $L(q,p)=\sqrt{1-e^2}$. We set
$e=0.6$ since, in this experiment, we are going to use constant
stepsize (see \cite[Sec. I.2.3]{HLW}). More precisely, we solve
problem \eqref{Hkepler} in the interval $[t_0, T]=[0, 50]$ by the
two-stages
 method \eqref{epgauss} with  the following set
of stepsizes: $h_i=2^{-i}$, $i=1,\dots,7$. Figure \ref{kepler_Ham}
reports the errors in the  Hamiltonian function $H$ and in the
angular momentum $L$ of the numerical solutions generated by the
method implemented with the intermediate stepsize $h=2^{-5}$. These
plots, which remain almost the same whatever is the stepsize
considered in the given range, testify that the integration
procedure performed by method \eqref{epgauss} is indeed feasible and
both  energy and angular momentum preservation may be recovered in
the discrete approximation of \eqref{hamilode}. For comparison
purposes, we also report the same quantities for the Gauss methods
of order $4$ (corresponding to the choice $\alpha=0$ in
\eqref{qgauss}).
\begin{figure}[htb]
\begin{center}
\includegraphics[width=12cm,height=7cm]{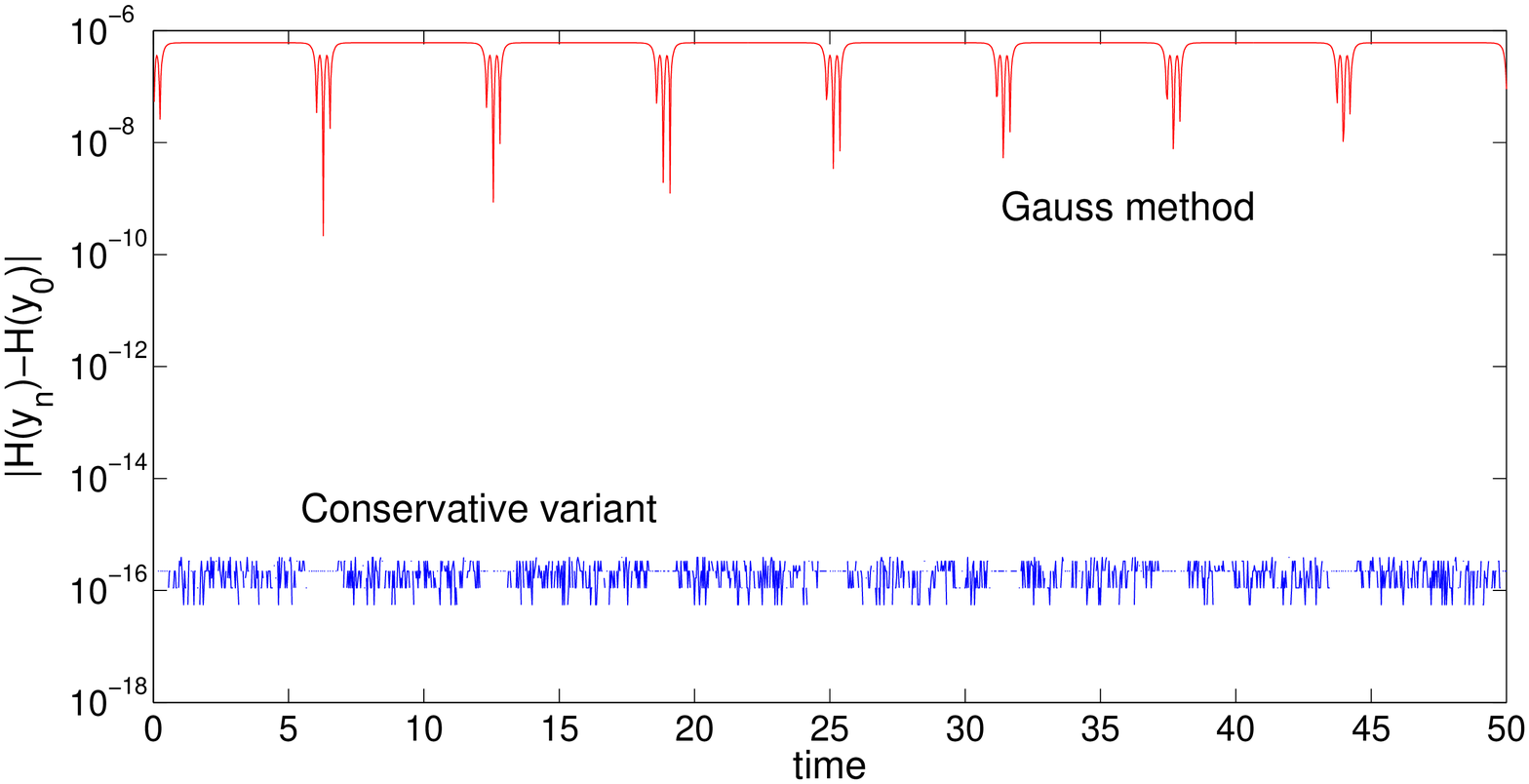}\\
\includegraphics[width=12cm,height=7cm]{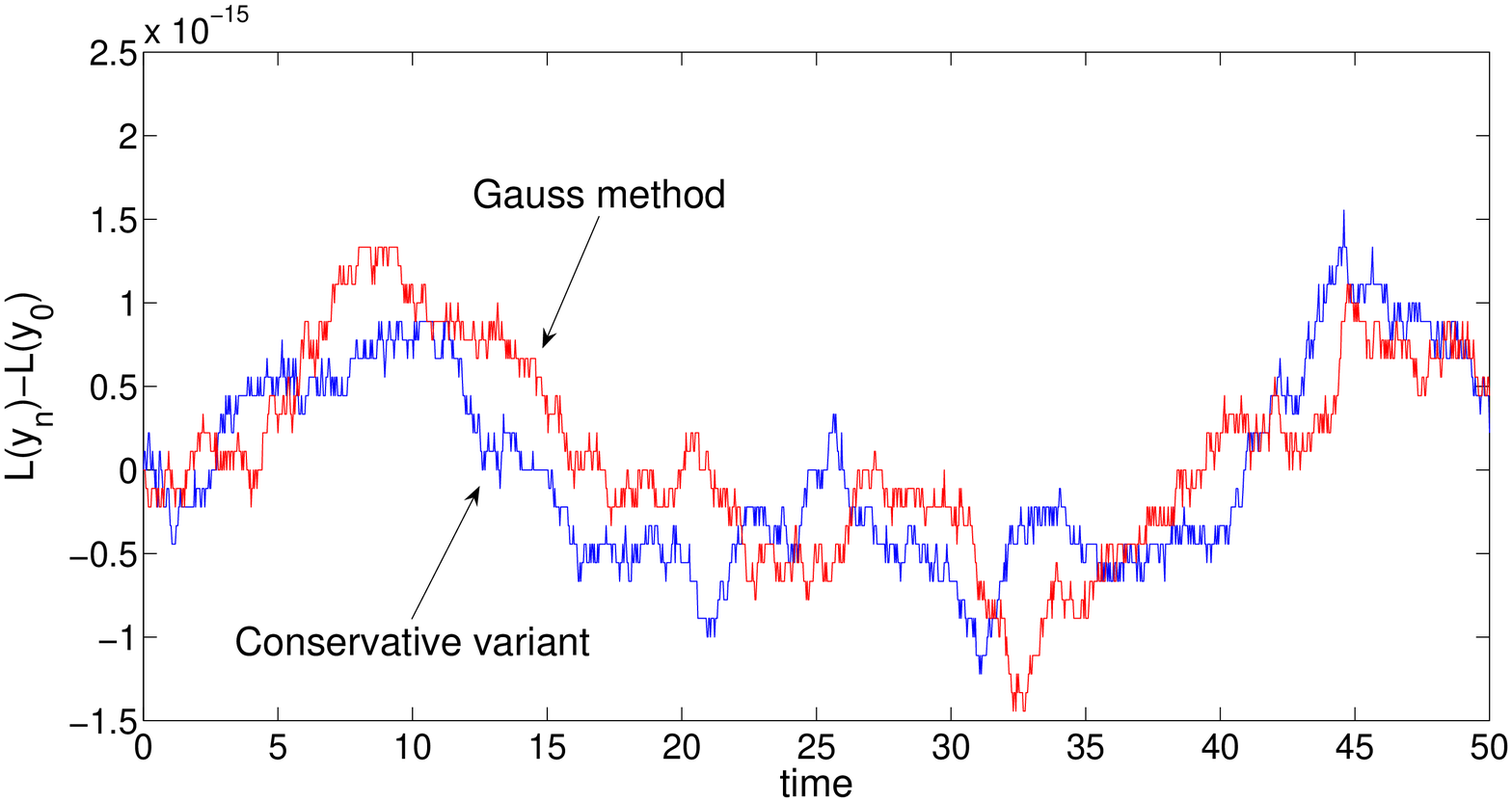}
\end{center}
\caption{Upper picture: errors in the Hamiltonian function of the
Kepler problem evaluated along the numerical solution generated by
the Gauss method of order four and its conservative variant (method
\eqref{epgauss} with $s=2$). Bottom plot: error in the numerical
angular momentum of the solution computed by the two methods. In
both cases the stepsize used is $h=2^{-5}$.} \label{kepler_Ham}
\end{figure}

The second and third columns of Table \ref{kepler_tab} report the
global error $e(h_i)=|y_N(h_i)-y(T)|$, $N=T/h_i$, at the end point
of the integration interval and the corresponding numerical order.
According to Theorem \eqref{fastorder}, we see that the maximum
order is preserved by method \eqref{epgauss}.

In Figure \ref{kepler_alpha_h5} the sequence $\alpha^\ast_n$,
corresponding to the values of the parameter $\alpha$ that at each
step restore the conservation of the energy, are plotted for the
case $h=2^{-5}$.  We consider
$\delta(h)=\max_n(\alpha^\ast_n)-\min_n(\alpha^\ast_n)$ as a measure
of the total variability of the values of the sequence $\{
\alpha^\ast_n\}$. Such quantity is reported in the fourth column of
Table~\ref{kepler_tab} for the values of the stepsize $h_i$ used in
this test. According to the result of Theorem \ref{implicit}, the
last column in the table confirms that the dependence of $\delta(h)$
on the stepsize $h$ is of the form $\delta = c h^2 +
\mathrm{h.o.t.}$, with $c \simeq 0.16$.

\begin{figure}[htb]
\begin{center}
\includegraphics[width=12cm,height=7cm]{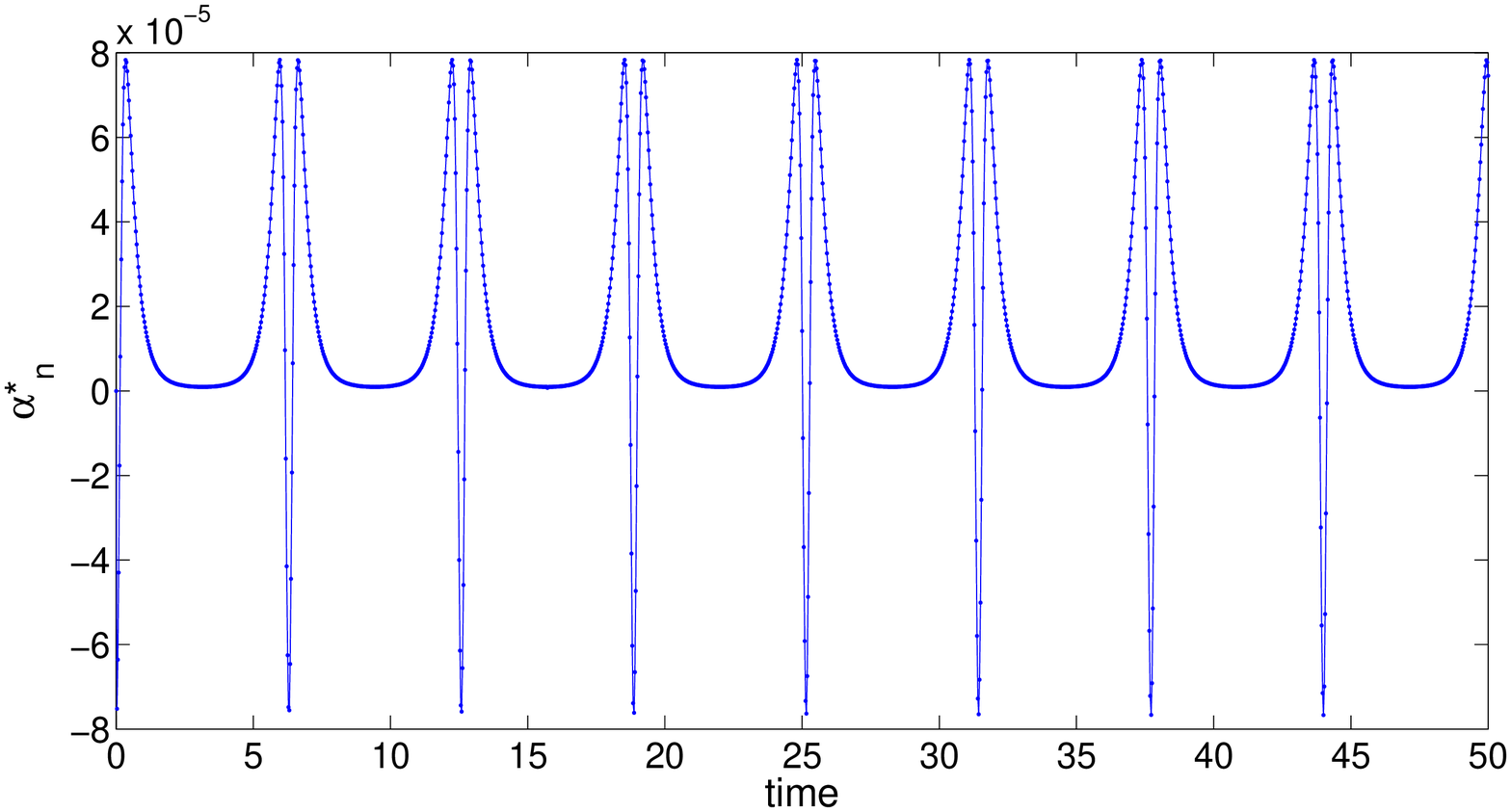}
\end{center}
\caption{Sequence of the values of the parameter $\alpha^\ast$ in
the  method \eqref{epgauss} with $s=2$ and $h=2^{-5}$.}
\label{kepler_alpha_h5}
\end{figure}

\begin{table}[hbt]
$$
\begin{array}{|c|cccc|}
\hline
h  &  e(h) &  \mbox{order} &  \delta(h) &  \delta(h)/h^2 \\
\hline
2^{-1}   &    2.62\cdot 10^{0}    \quad & \quad        \quad & \quad 2.13\cdot 10^{-2} \quad & \quad  8.5374\cdot 10^{-2}  \\[.2cm]
2^{-2}   &    3.85\cdot 10^{-1}   \quad & \quad  2.763 \quad & \quad 1.04\cdot 10^{-2} \quad & \quad  1.6700\cdot 10^{-1}  \\[.2cm]
2^{-3}   &    2.50\cdot 10^{-2}   \quad & \quad  3.945 \quad & \quad 2.52\cdot 10^{-3} \quad & \quad  1.6185\cdot 10^{-1}  \\[.2cm]
2^{-4}   &    1.59\cdot 10^{-3}   \quad & \quad  3.970 \quad & \quad 6.23\cdot 10^{-4} \quad & \quad  1.5951\cdot 10^{-1}  \\[.2cm]
2^{-5}   &    1.00\cdot 10^{-4}   \quad & \quad  3.991 \quad & \quad 1.55\cdot 10^{-4} \quad & \quad  1.5878\cdot 10^{-1}  \\[.2cm]
2^{-6}   &    6.28\cdot 10^{-6}   \quad & \quad  3.997 \quad & \quad 3.87\cdot 10^{-5} \quad & \quad  1.5862\cdot 10^{-1}  \\[.2cm]
2^{-7}   &    3.93\cdot 10^{-7}   \quad & \quad  3.999 \quad & \quad 9.67\cdot 10^{-6} \quad & \quad  1.5856\cdot 10^{-1}  \\[.2cm]
\hline
\end{array}
$$
\caption{Performance of the order four method \eqref{epgauss}
applied to the Kepler problem. The  global error  at $T=50$ (second
column), and the corresponding order obtained via the formula
$\log_2(e(h_i)/e(h_{i+1}))$, indicate that the perturbations
introduced in the Gauss collocation conditions (see \eqref{coll})
are small enough that the order $4$ of the Gauss method with two
stages is conserved by its energy preserving variant. The last two
columns give a measure of the perturbations and of the rate they
tend to zero as $h\rightarrow 0$. The quantity $\delta(h)$ is the
amplitude of the minimum interval that encloses all the values
$\alpha^\ast_n$ for the given stepsize $h$ and in the given
integration interval. Hence the last column confirms what proved in
Theorem \ref{implicit}, namely that the perturbations are $O(h^2)$.}
\label{kepler_tab}
\end{table}

\subsection{Test problem 2}
We consider the problem defined by the following polynomial
Hamiltonian function:
\begin{equation}
\label{macie1H} H(q_1,q_2,p_1,p_2)=\frac{1}{2}(p_1^2+p_2^2) +
(q_1^2+q_2^2)^2.
\end{equation}
This problem has been proposed in \cite{MaPr} as an example of a
class of polynomial systems which, under suitable assumptions, admit
an additional polynomial first integral $F$ which is functionally
independent from $H$. In this case, the additional (irreducible)
first integral is
\begin{equation}
\label{macie1L} L(q_1,q_2,p_1,p_2)=q_1p_2-q_2p_1.
\end{equation}
The polynomial $L$ being quadratic, we expect that our methods may
preserve both $H$ and $L$.\footnote{Of course $L$ may again be
interpreted as the angular momentum of a mechanical system having
\eqref{macie1H} as Hamiltonian function.}

We have solved problem \eqref{macie1H} by means of two methods of
order six ($s=3$): method \eqref{epgauss}, and the order six method
of the second type, described in Theorem \ref{implicit1}.

Figure \ref{fig_macie1} reports the errors in the Hamiltonian
function $H$ and in the quadratic first integral $L$ of the
numerical solutions generated by the latter method implemented with
the intermediate stepsize $h=2^{-3}$. For comparison purposes, we
also report the same quantities for the Gauss methods of order six.

\begin{figure}[htb]
\begin{center}
\includegraphics[width=12cm,height=7cm]{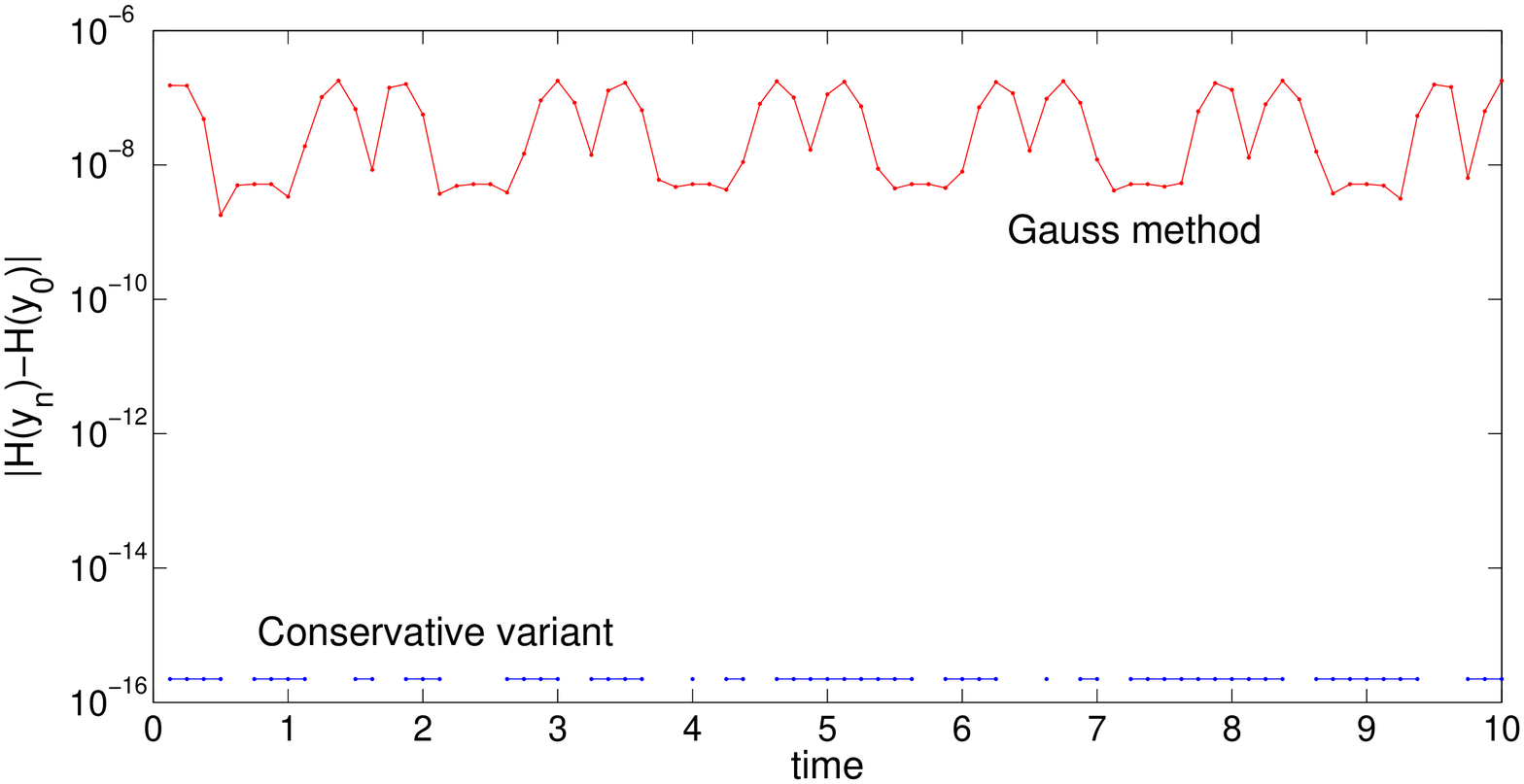}\\
\includegraphics[width=12cm,height=7cm]{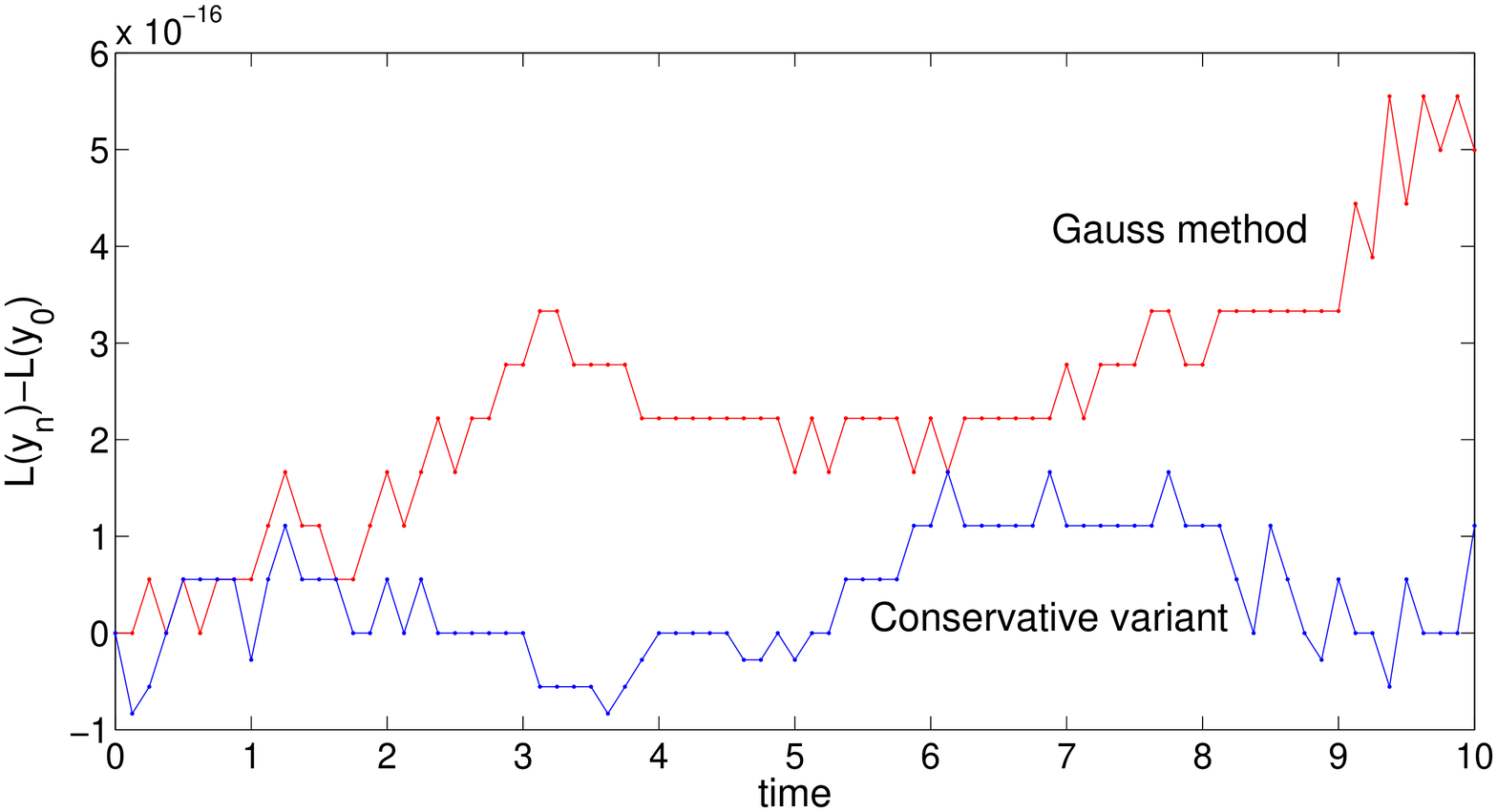}
\end{center}
\caption{Upper picture: errors in the Hamiltonian function of test
problem 2 evaluated along the numerical solution generated by the
Gauss method of order six and its conservative variant of the second
type. Bottom plot: error in the quadratic first integral
\eqref{macie1L} of the solution computed by the two methods. In both
cases the stepsize used is $h=2^{-3}$.} \label{fig_macie1}
\end{figure}

Tables \ref{macie1_tab1} and \ref{macie1_tab2} are the analogues of
Table \ref{kepler_tab} for these two methods: we see that both
methods achieve order six but, while in the former
$\alpha^\ast(h)=O(h^2)$, in the latter $\alpha^\ast(h)=O(h^4)$
consistently with Theorems \ref{implicit}, \ref{fastorder}, and
\ref{implicit1}.
\begin{table}[hbt]
$$
\begin{array}{|c|cccc|}
\hline
h  &  e(h) &  \mbox{order} &  \delta(h) &  \delta(h)/h^2 \\
\hline
2^{-1}   &        2.17\cdot 10^{-2}    \quad & \quad        \quad & \quad     1.59\cdot 10^{-2} \quad & \quad      6.37\cdot 10^{-2}  \\[.2cm]
2^{-2}   &        4.59\cdot 10^{-4}    \quad & \quad      5.562 \quad & \quad     3.99\cdot 10^{-3} \quad & \quad      6.39\cdot 10^{-2}  \\[.2cm]
2^{-3}   &        7.77\cdot 10^{-6}    \quad & \quad      5.884 \quad & \quad     9.99\cdot 10^{-4} \quad & \quad      6.40\cdot 10^{-2}  \\[.2cm]
2^{-4}   &        1.24\cdot 10^{-7}    \quad & \quad      5.970 \quad & \quad     2.53\cdot 10^{-4} \quad & \quad      6.48\cdot 10^{-2}  \\[.2cm]
2^{-5}   &        1.94\cdot 10^{-9}    \quad & \quad      5.992 \quad & \quad     6.33\cdot 10^{-5} \quad & \quad      6.49\cdot 10^{-2}  \\[.2cm]
2^{-6}   &        3.05\cdot 10^{-11}   \quad & \quad      5.994 \quad & \quad     1.59\cdot 10^{-5} \quad & \quad      6.51\cdot 10^{-2}  \\[.2cm]
\hline
\end{array}
$$
\caption{Performance of method \eqref{epgauss} of order six applied
to  problem \eqref{macie1H}. The reported quantities are the
analogues of the ones presented in Table \ref{kepler_tab}.}
\label{macie1_tab1}
\end{table}
\begin{table}[hbt]
$$
\begin{array}{|c|cccc|}
\hline
h  &  e(h) &  \mbox{order} &  \delta(h) &  \delta(h)/h^4 \\
\hline
2^{-1}   &            4.91\cdot 10^{-2}    \quad & \quad        \quad & \quad                     5.59\cdot 10^{-2} \quad & \quad          0.895 \\[.2cm]
2^{-2}   &            1.46\cdot 10^{-2}    \quad & \quad          1.753 \quad & \quad             1.51\cdot 10^{-2} \quad & \quad          3.87  \\[.2cm]
2^{-3}   &            1.84\cdot 10^{-4}    \quad & \quad          6.304 \quad & \quad             4.92\cdot 10^{-4} \quad & \quad          2.01  \\[.2cm]
2^{-4}   &            3.23\cdot 10^{-6}    \quad & \quad          5.836 \quad & \quad             4.07\cdot 10^{-5} \quad & \quad          2.66  \\[.2cm]
2^{-5}   &            4.73\cdot 10^{-8}    \quad & \quad          6.091 \quad & \quad             2.30\cdot 10^{-6} \quad & \quad          2.41  \\[.2cm]
2^{-6}   &            7.03\cdot 10^{-10}   \quad & \quad          6.074 \quad & \quad             1.50\cdot 10^{-7} \quad & \quad          2.51  \\[.2cm]
\hline
\end{array}
$$
\caption{Performance of the sixth-order method of the second kind
applied to problem \eqref{macie1H}.} \label{macie1_tab2}
\end{table}

\subsection{The H\'enon-Heiles problem}
The H\'{e}non-Heiles equation originates from a problem in Celestial
Mechanics describing the motion of a star under the action of a
gravitational potential of a galaxy which is assumed
time-independent and with an axis of symmetry (the $z$-axis) (see
\cite{HH} and references therein). The main question related to this
model was to state the existence of a third first integral, beside
the total energy and the angular momentum. By exploiting the
symmetry of the system and the conservation of the angular momentum,
H\'enon and Heiles reduced from three (cylindrical coordinates) to
two (planar coordinates) the degrees of freedom, thus showing that
the problem was equivalent to the study of the motion of a particle
in a plane subject to an arbitrary potential $U(q_1,q_2)$:
\begin{equation}
\label{HH}
H(q_1,q_2,p_1,p_2)=\frac{1}{2}(p_{1}^2+p_{2}^2)+U(q_1,q_2).
\end{equation}

\begin{figure}[t]
\begin{center}
\includegraphics[width=12cm,height=7cm]{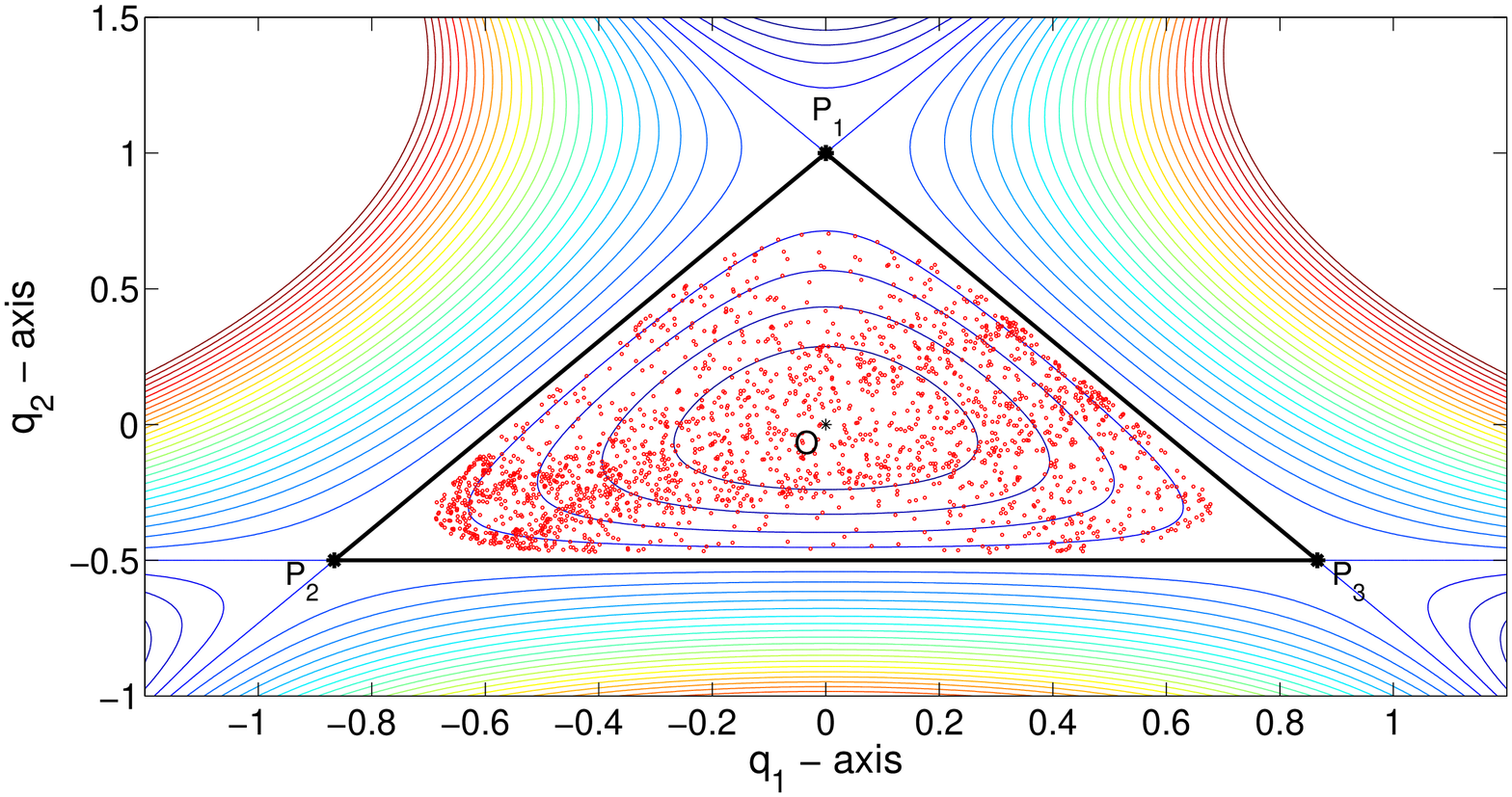}
\end{center}
\caption{Level curves of the potential $U(q_1,q_2)$ of the
H\'enon-Heiles problem (see \eqref{Henon_potential}). The origin $O$
is a stable equilibrium point, whose domain of stability contains
the equilateral triangle having as vertices the saddle points $P_1$,
$P_2$, and $P_3$, provided that the total energy does not exceed the
value $\frac{1}{6}$. Inside the triangle, a numerical trajectory
(small dots) computed by  the sixth-order method of the second type
with stepsize $h=0.25$ and in the time interval $[0, 500]$, is
traced: its total energy is $0.15$.}\label{henon_fig1}
\end{figure}

In particular, for their experiments they chose
\begin{equation}
\label{Henon_potential}
U(q_1,q_2)=\frac{1}{2}(q_{1}^2+q_{2}^2)+q_{1}^2q_{2}-\frac{1}{3}q_{2}^3,
\end{equation}
which makes the Hamiltonian function a polynomial of degree three.
When $U(q_1,q_2)$ approaches the value $\frac{1}{6}$, the level
curves of $U$ tend to an equilateral triangle, whose vertices are
saddle points of $U$ (see Figure \ref{henon_fig1}). This vertices
have coordinates $P_1=(0,1)$,
$P_2=(-\frac{\sqrt{3}}{2},-\frac{1}{2})$ and
$P_3=(\frac{\sqrt{3}}{2},-\frac{1}{2})$.

Since $U$ in \eqref{HH} has no symmetry in general, we cannot
consider the angular momentum as an invariant anymore, so that the
only known first integral is the total energy represented by
\eqref{HH} itself, and the question is whether or not a second
integral does exist. H\'enon and Heiles conducted a series of tests
with the aim of giving a numerical evidence of the existence of such
integral for moderate values of the energy $H$, and of the
appearance of chaotic behavior when $H$ becomes larger than a
critical value: it is believed that for values of $H$ in the
interval $(\frac{1}{8},\frac{1}{6})$ this second first integral does
not exist (see also \cite[Section I3]{HLW}).

We consider the initial point
$P_0=(q_{10},\,q_{20},\,p_{10},\,p_{20})=(0,\, 0,\,
\sqrt{\frac{3}{10}},\, 0)$ which confers on the system a total
energy $H=0.15 \in(\frac{1}{8},\frac{1}{6})$. Therefore the orbit
originating from $P_0$ will never abandon the triangle for any value
of the time $t$. We have integrated problem \eqref{HH} in the time
interval $[0,\,500]$ with stepsize $h=0.25$ by using the Gauss
method of order six  and its conservative variant of the second
type. Figure \ref{fig2.henon} shows the errors in the Hamiltonian
function $H$ in both cases.
\begin{figure}[ht]
\begin{center}
\includegraphics[width=13cm,height=8cm]{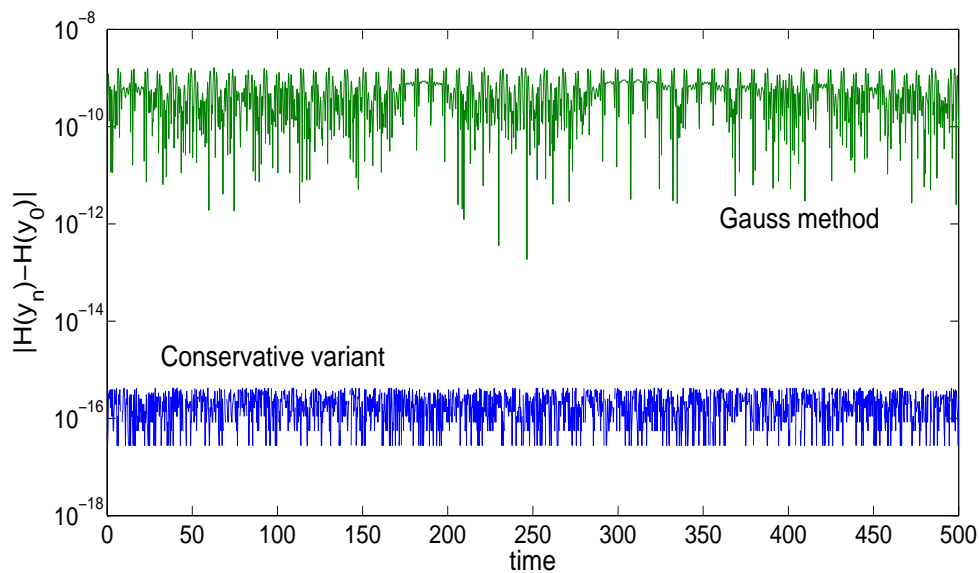}
\end{center}
\caption{Errors in the Hamiltonian function
\eqref{HH}-\eqref{Henon_potential} evaluated along the numerical
solution generated by the Gauss method of order six and its
conservative variant of the second type. Stepsize: $h=0.25$; time
interval: $[0, 500]$; initial condition
$(q_{10},\,q_{20},\,p_{10},\,p_{20})=(0,\, 0,\,
\sqrt{\frac{3}{10}},\, 0)$.} \label{fig2.henon}
\end{figure}

\section{Conclusions}
We have defined a new class of symmetric and symplectic one-step
methods of any high order that, under somewhat weak assumptions,
are capable to compute a numerical solution along which the
Hamiltonian function is precisely conserved. This feature has been
realized by first introducing a symplectic parametric perturbation
of the Gauss method, and then by selecting the parameter, at each
step of the integration procedure, in order to get energy
conservation. A relevant implication of the symplectic nature of
each formula is the conservation of all quadratic first integrals
associated to the system. With the help of the implicit function
theorem, we have shown that not only do these methods exist, but
that the correction required on the Gauss method is so small that
the order of convergence of this latter method is  preserved by
its conservative variant. A few test problems have been reported
to confirm the theoretical results presented, and to show the
effectiveness of the new formulae.

This approach opens a number of interesting routes of investigation.
First of all, if preferred, the parameter could be selected in such
a way to impose the conservation of other non quadratic first
integrals different from the Hamiltonian function itself. More
generally, the multi-parametric generalization introduced suggests
the possibility of choosing the free parameters in order to impose
the conservation of a number of functionally independent first
integrals possessed by the continuous problem. Last but not least,
the idea of considering symplectic corrections of the Gauss method
could  be in principle  extended to other classes of symplectic
methods known in the literature. The above described lines of
investigation, as well as the efficient solution of the nonlinear
systems arising from the conservation requirements, will be the
subject of future researches.

\end{document}